\documentclass[12pt]{article}
\usepackage{amsmath}
\usepackage{amsfonts}
\usepackage{amssymb}
\usepackage{latexsym}
\usepackage{curves}
\usepackage{color}
\usepackage{enumerate}
\usepackage{amsthm}

\begin{document}
\baselineskip 16pt

\def \mod{\hbox{\ \rm mod }}
\def \Aut {\mbox{\rm Aut}}
\def \min {\mbox{\rm min}}
\def\M{\mathcal{M}}
\newcommand{\ZZ}{\mathbb{Z}}

\newtheorem{theorem}{Theorem}[section]
\newtheorem{lemma}[theorem]{Lemma}
\newtheorem{prop}[theorem]{Proposition}
\newtheorem{cor}[theorem]{Corollary}
\theoremstyle{remark}
\newtheorem{remark}[theorem]{Remark}

\title{Classification of reflexibile  regular Cayley maps  for dihedral groups}
\author{Istv\'an Kov\'acs,$^{a}$   \; Young Soo Kwon$^{\, b}$ \\  [+0.75ex]
$^a$ {\small FAMNIT \& IAM, University of Primorska, Muzejski trg 2, 6000 Koper, Slovenia} \\ [-0.5ex]
$^b$ {\small Mathematics, Yeungnam University, Kyongsan 712-749, Republic of Korea}
}
\date{}
\maketitle
\let\thefootnote\relax\footnote{
The work was partially supported by the Slovenian-Korean bilateral
project, grant no.\ BI-KOR/13-14-002.  The first author was
supported in part by the Slovenian Research Agency (research program
P1-0285 and research projects N1-0032, J1-5433, and J1-6720). The
second author was supported by Basic Science Research Program
through the National Research Foundation of Korea(NRF) funded by the
Ministry of Education, Science and Technology
(2010-0022142).  \\
[+0.5ex] {\it  E-mail addresses:} istvan.kovacs@upr.si (Istv\'an
Kov\'acs), ysookwon@ynu.ac.kr (Young Soo Kwon).
}

\begin{abstract}
In this paper, we classify reflexible regular Cayley maps for dihedral groups.
 \\

\end{abstract}

\section{ Introduction}

Let $G$ be a finite group, and let $X \subseteq G\!\setminus\!\{1_G \}$ be
a  generating set for $G$ which is closed under inverse. The {\em Cayley graph} $C(G,X)$ for the pair $(G,X)$ has vertex set $G$, with any two vertices
$g,h \in G$ joined by an edge whenever $g^{-1}h \in X$. It follows that the Cayley graphs considered in this article are finite, connected, undirected and simple.

A (topological) \emph{map} is a 2-cell embedding of a connected graph into a closed surface. We are interested in particular embeddings of Cayley graphs on
surfaces. It is well known that to describe an embedding of graph into an orientable surface, one just needs to specify, at every vertex, a cyclic ordering of edges
emanating from the vertex.  A graph automorphism which can be extended
to a homeomorphism of the supporting surface onto itself is called
a {\em map automorphism}. If the group of orientation-preserving automorphisms of
an embedding acts transitively on incident vertex-edge pairs called \emph{arcs}, then
this action is also regular. In this case,  the embedding (map) is called {\em regular}.
In addition, if there is an orientation-reversing map automorphism,  the map is called {\em reflexible}.

For a Cayley graph $C(G,X)$, all edges incident to a
vertex $g\in G$ have the form $\{g,gx\}$ where $x\in X$.
Hence to describe an embedding of a Cayley graph in an
orientable surface, it suffices to specify a cyclic permutation of
the generating set at each vertex. If these cyclic permutations are the same at each vertex, which is a cyclic permutation $p$ of the set
$X$, then the embedding is called a {\em Cayley map} and is
denoted by $CM(G,X,p)$. In this paper, when we describe $p$ clearly, we sometimes
omit the description of $X$ because $p$ contains full information of $X$.
Since left multiplication by any fixed element of $G$ induces an
orientation-preserving automorphism of $CM(G,X,p)$, Cayley maps are
vertex-transitive. So a Cayley map $CM(G,X,p)$ is regular if and
only if the stabilizer $\Aut(\M)_{1_G}$  of the vertex $1_G$ in the
group $\Aut(\M)$ acts regularly on the arcs emanating from $1_G$. A
generator of $\Aut(\M)_{1_G}$  satisfies a certain system of
identities \cite{JA1}, subsequently re-stated in terms of so-called
skew-morphisms \cite{JAS} which will be introduced later. For more
information on regular Cayley maps, the reader is referred to
\cite{BI,JAS,RS+,SS1,SS2}.

For a Cayley map $\M = CM(G,X,p)$, $\M$ is called \emph{$t$-balanced} if $p(x)^{-1} = p^t(x^{-1})$ for every $x \in X$.
In particular, if $t=1$ then $\M$ is called \emph{balanced}, and if $t=-1,$ then $\M$ is said to be \emph{anti-balanced}.
Note that a Cayley map ${\cal M} =
CM(G,X,p)$ is regular and balanced if and only if there exists a
group automorphism $\psi$ of $G$ whose restriction on $X$ is $p$.
Two Cayley maps $CM(G_1,X_1,p_1)$ and $ CM(G_2,X_2,p_2)$ are said
to be {\em equivalent} if there exists a group isomorphism $\phi :
G_1 \to G_2$ mapping $X_1$ to $X_2$ such that $\phi p_1=p_2\phi$.
Equivalent Cayley maps are isomorphic as maps. On the other hand, isomorphic Cayley maps may not be equivalent as Cayley maps.

The class of cyclic groups is the only class of finite groups on which all regular Cayley maps have been classified due to the work of Conder and Tucker \cite{CT}. Regarding other groups, only partial classifications are known (see, e.g.\ \cite{KKF,KMM,KO,WF}).
In this paper we investigate in detail properties of reflexible
Cayley maps for dihedral groups. As the main result, we classify reflexible regular Cayley maps for dihedral groups. The following theorem is the main result.

\begin{theorem} \label{main-result}
Let  $\M = CM(D_n, X, p)$ be a reflexible regular Cayley map for the
dihedral group $D_n$ of valency $d,$ where $D_n=\left< a,b \mid
a^n=b^2=1, bab=a^{-1} \right>$ is a dihedral group of order $2n$. If
$d=2$ then $\M$ is isomorphic to  $CM(D_n, \{b,ab\}, (b, ab))$, and
if $d=3$ then $\M$ is isomorphic to
$CM(D_2,\{b,ab,a\},(b,ab,a))$, $CM(D_3,\{b,ab,a^2b\},(b,ab,a^2b))$
or $CM(D_4,\{b,$ $ a,a^{-1} \},(b,a,a^{-1}))$. For $d \ge 4$, $\M$
is isomorphic to one of the maps in the following list:
\begin{enumerate}
\item[$(1)$] ${\cal M}_1 = CM(D_n, X_1, p_1)$ with
$$ p_1 = (b, ab, a^{\ell+1}b,  a^{\ell^2+\ell+1}b,  \ldots, a^{\ell^{d-2}+\ell^{d-3}+\cdots+1}b)$$ for some $\ell$ satisfying $\ell^2 \equiv 1 (\mod n)$, where $d$ is the smallest positive integer such that $\ell^{d-1}+\ell^{d-2}+\cdots+1 \equiv 0 (\mod n)$.
\item[$(2)$] A positive integer $n$ is even and ${\cal M}_2 = CM(D_n, X_2, p_2)$ with $$ p_2 = (b,a,a^2 b,a^3,a^4 b,\ldots,a^{n-2}b,a^{n-1}).$$
\item[$(3)$] A positive integer $n=2m$ is a multiple of $8$  and ${\cal M}_3 = CM(D_n, X_3, p_3)$ with $$ p_3 = (b,a,a^{m+2} b,a^3,a^4 b,\ldots,a^{m-2}b,a^{n-1}).$$
\item[$(4)$]  $n=3k$ for some positive integer $k$ and ${\cal M}_4 = CM(D_n, X_4, p_4)$ with $$ p_4 = (a^{-1},a, b,a^2b).$$
\item[$(5)$] $n=2m=8k+4$ for some positive integer $k$ and ${\cal M}_5 = CM(D_n, X_5, p_5)$ with $$ p_5 = (a^{-1},a, b,a^{m+1}, a^{m-1}, a^{m+2}b).$$
\item[$(6)$] $n=2m=4k+2$ for some positive integer $k$ and ${\cal M}_6 = CM(D_n, X_6, p_6)$ with $$ p_6 = (b,a^{m-2}b, a, a^{-2} b, a^{m}b, a^{-1}).$$
\end{enumerate}
Furthermore all the Cayley maps listed above are reflexible regular Cayley maps.
\end{theorem}

\section{Preliminaries}

Let $G$ be a finite group. Consider a permutation $\varphi$ of $G$
of order $d$ (in the full symmetric group Sym$(G)$) and a function
$\pi$ from $G$ to the cyclic group $\ZZ_d$. The function $\varphi$
is said to be a {\em skew-morphism} of $G$ with the associated {\em
power function} $\pi$, if $\varphi$ fixes the identity element of
$G$ and \[ \varphi(ab)= \varphi(a)\varphi^{\pi(a)}(b) \ \ \ {\rm
for\ all\ \ } a,b\in G.\] Here $\varphi^j$ stands for the
composition $\varphi\circ\dots \circ \varphi$ consisting of $j$
terms.  It was  proved in \cite{JAS} that a Cayley map $CM(G,X,p)$
is regular if and only if there exists a skew-morphism $\varphi$ of
$G$ such that $\varphi(x)= p(x)$ for each $x\in X$. For a given
finite group $G$ with a symmetric generating set $X \subseteq
G\!\setminus\!\{1_G \}$, and a given regular Cayley map ${\cal M} =
CM(G, X, p)$, we will always use $\varphi$ and $\pi$ to denote
respectively the skew-morphism of ${\cal M}$ with the property that
the restriction of $\varphi$ to $X$ is $p$, and the associated power
function. We will use $d$ to denote the valency of ${\cal M}$, and
identify the integers $0, 1, \ldots, d-1$ with their residue classes
modulo $d$ when the context permits. Also, if $p = (x_0,x_1,\dots,
x_{d-1})$, then we will let $c(k) \in \ZZ_d$ be the subscript of
$x_k^{-1}$, and so $x_k^{-1}=x_{c(k)}$ and $c(c(k)) = k$ for all $k
\in \ZZ_d$.  Now for all $k \in \ZZ_d$ we have $1_G =
\varphi(1_G) = \varphi(x_{k}x_{k}^{-1}) =
\varphi(x_{k})\varphi^{\pi(x_k)}(x_{k}^{-1}) =
x_{k+1}\varphi^{\pi(x_k)}(x_{c(k)})$, which implies
$\varphi^{\pi(x_k)}(x_{c(k)}) = x_{k+1}^{-1} = x_{c(k+1)}$; and thus
$\pi(x_k) = c(k+1) - c(k)$ holds in $\ZZ_d$
 for all $k \in \ZZ_d$. It follows from \cite{JAS} that all values of a power function
$\pi$ are non-zero when the skew-morphism associated with $\pi$ is
not equal to the identity mapping. The \emph{kernel} $\ker(\pi)$ of
the power function $\pi$ is defined by $\ker(\pi) = \{g \in G :
\pi(g) = 1\}$.  The following proposition gives some properties of
skew-morphisms.

\begin{prop}{\rm \cite{JAS}} \label{properties-skew}
Let $\varphi$ be a skew-morphism of $G$ with the associated power function $\pi$.
Then
\begin{enumerate}
\item [$(1)$] $\ker(\pi) = \{ g \in G \mid \pi(g)=1 \}$ is a
subgroup of $G$. \item [$(2)$] $\pi(g) = \pi(h)$ if and only if $g$
and $h$ are contained in the same right coset of the subgroup
$\ker(\pi)$.
\item [$(3)$] For any $g,h \in G$, $\pi(gh)=\sum_{i=0}^{\pi(g)-1} \pi(\varphi^i (h))$.
\end{enumerate}
\end{prop}

Two Cayley maps ${\cal M}_1 = CM(G,X,p)$ and ${\cal M}_2 =
CM(G,Y,q)$ with  $X=\{x_0,x_1, \ldots,$ $x_{d-1}\}$ and
$Y=\{y_0,y_1,\ldots,y_{d-1}\}$ are said to have the {\it same
rotation type} if the two cyclic permutations $p$ and $q$ can be
expressed as  $p = (x_0 \;x_1\; \cdots \; x_{d-1})$ and
$q = (y_0 \;y_1\; \cdots \; y_{d-1})$ such that $x_i^{-1} =
p^{j}(x_i)$ if and only if $y_i^{-1} = q^{j}(y_i)$ for any
 $i,j \in \{0,1,\ldots,d-1\}$.

\begin{prop} {\rm \cite{KKF}} \label{iso}
Let ${\cal M}_1 = CM(G,X,p)$ and ${\cal M}_2 = CM(G,Y,q)$ be two
regular Cayley maps of the same  rotation type. Then ${\cal M}_1 $
and ${\cal M}_2$ are isomorphic if and only if there exists a
group automorphism $\varphi$ of $G$ such that $\varphi(X) = Y$ and $\varphi
p = q \varphi$, i.e., they are equivalent.
\end{prop}

The following observation will be useful later on.

\begin{prop} {\rm  \cite{CKS}} \label{refle-images}
Let ${\cal M} = CM(G, X, p )$ with $p =
(x_0,x_1,\dots,x_{d-1})$,  be a $d$-valent reflexible Cayley map
for a finite group $G$, and let $\psi$ be an
orientation-reversing automorphism of ${\cal M}$. If $\psi$ takes
the arc $(g,gx_i)$ to the arc $(h,hx_j)$, then $\psi$ takes
$gx_{i}x_{k}$ to $hx_{j}x_{c(j)+c(i)-k}$ for all $k \in \ZZ_d$.
\end{prop}

For a given Cayley map ${\cal M} = CM(G, X, p )$, a bijection $\tau$ from $G$ to $G$ is called an {\em antirotary mapping} of ${\cal M}$ if the following three conditions hold
for any $g \in G$ and $x \in X$:
$$ {\rm (i)} \ \tau(1_G)=1_G\ \  {\rm (ii)}  \ \tau(g)^{-1}\tau(gx) \in X \ \ \mbox{and}
\ \ {\rm (iii)} \ \tau(g)^{-1}\tau(gp(x)) = p^{-1}(\tau(g)^{-1}\tau(gx)).$$
In fact, an antirotary mapping of ${\cal M}$ is an orientation-reversing map automorphism fixing $1_G$ and the converse is also true \cite{JA1}. So the following proposition holds.

\begin{prop} {\rm \cite{JA1}}  \label{antirotary}
Let ${\cal M} = CM(G, X, p )$  be a regular Cayley map. Now ${\cal
M}$ is reflexible if and only if there exists an antirotary mapping
of ${\cal M}$.
\end{prop}

For any $n$, let $D_n = \left<  a, b \mid a^n = b^2 = 1
~~\mbox{and}~~ab=ba^{-1} \right>$ be the dihedral group of order
$2n$. For our convenience, let
$$
A_n = \{a^i \ | \ i=1,\ldots,n\} \ \ \mbox{and} \ \ B_n = \{ a^ib \ | \ i=1,\ldots,n\}.
$$
Now $A_n$ is the cyclic subgroup of $D_n$ of order $n$. The automorphism group  of $D_n$ is
$$
\Aut(D_n) = \{ \sigma_{i,j} \mid \sigma_{i,j}(a) = a^i, ~ \sigma_{i,j}(b) =
 a^jb, ~~i,j \in \{1,2,\ldots,n \} ~~{\rm and~~} \gcd(i,n) = 1 \}.$$

For any integer $\ell$ with $\gcd(\ell,n)=1$, let $d$ be the
smallest positive integer such that
$$\ell^{d-1}+\ell^{d-2}+\cdots+\ell+1 \equiv 0\;(\mod n).$$ Set
$X=\{b,ab,a^{\ell+1}b,\ldots,a^{\ell^{d-2}+\ell^{d-3}+\cdots+\ell+1}b
\},$ $p=(b, \;ab,\; a^{\ell+1}b, \; \cdots ,\;
a^{\ell^{d-2}+\ell^{d-3}+\cdots+\ell+1}b),$ and $\M(n,\ell)
=CM(D_n,X,p)$ for the resulting regular Cayley map. Wang and
Feng~\cite{WF} showed that all regular balanced Cayley maps on $D_n$
are those of the form $\M(n,\ell),$ where $\gcd(n,\ell)=1$,
and any two such maps on $D_n$ with two integers $\ell_1$ and $\ell_2$ such
that $\gcd(\ell_1,n)=\gcd(\ell_2,n)=1$ are isomorphic if and only
if $\ell_1=\ell_2$.
\medskip

Let $\M = CM(G,X,p)$ be a regular Cayley map. Suppose, in addition,
that there exists a subgroup $N \le G$ such that $N$  is normal in
$G$ and the the set of $N$-cosets is a block system of $\Aut(\M)$.
In what follows it will be simply said that $G/N$ is a block system
for $\Aut(\M)$. Furthermore, we set $X/N = \{ N x \mid x \in X \}$.
Clearly, $X/N$ is a generating subset of the factor group $G/N$ and
$X/N = (X/N) ^{-1}$.
 Also, since $C(G,X)$ is $\Aut(\M)$-arc-regular, no element of $X$ belongs to $N$, and so
$1_{G/N} \notin X/N$.

There is an action of $\Aut(\M)$ on the set of blocks, i.\ e., on
$G/N$. For $g \in \Aut(\M),$ we let $g^{G/N}$ denote the action of
$g$ on $G/N$, and for a subgroup $H \le \Aut(\M)$ set $H^{G/N} =\{
g^{G/N} \mid g \in H \}$. For $g \in G,$ let $L_g$ be the \emph{left
translation} $L_g : x \mapsto g x, x \in G,$ and for a subgroup $H
\le G,$ let $L(H)=\{L_h \mid h\in H \}$. Notice that $(L_g)^{G/N} =
L_{N g}$ for every $g \in G$. Let us write  $X = \{x_0,
x_1,\ldots,x_{d-1}\}$ and $p = (x_0,x_1,\ldots,x_{d-1})$. Then it
follows that the cycle $p^{G/N} : = (x_0 N,x_1
N,\ldots,x_{d-1} N)$ is well-defined (see \cite{KMM}); and so is
the Cayley map $CM(G/N,X/N,p^{G/N})$. The latter map is called  the
\emph{quotient} of $\M$ with respect to the block system $G/N,$ and
it will be also denoted by $\M/N$. We note that  the quotient map
$\M/N$ coincides with the so called \emph{Cayley-quotient} induced
by the normal subgroup $N$ which was defined by Zhang \cite{Z}, and
in the same paper $\M$ is also referred to as the
\emph{Cayley-cover} of $\M/N$.
 We collect below some properties (see \cite[Corollary~3.5]{KMM}):

\begin{prop}\label{quo}
Let $\M = CM(G,X,p)$ be a regular Cayley map with associated
skew-morphism $\psi$ and power function $\pi,$ and let $N \le  G$ be
a normal subgroup in $G$ and $G/N$ is a block system for $\Aut(\M)$.
Then the following hold:
\begin{enumerate}[(1)]
\item $\M/N = CM(G/N,X/N,p^{G/N})$ is also regular.
\item $\Aut(\M/N) = \Aut(\M)^{G/N}$.
\item The skew-morphisms associated  with $\M/N$ is equal to $\psi^{G/N}$.
\item The order
$|\langle \psi \rangle| \le |N| \cdot \big|\langle \psi^{G/N}
\rangle\big|,$  and equality holds if and only if $X$ is a union of
$N$-cosets.
\item The power function $\pi^{G/N}$ associated with $\M/N$ satisfies
$$
\pi^{G/N}(Ng) \equiv \pi(g)\big(\text{mod } \big|\langle \psi^{G/N}
\rangle\big| \big) \text{ for every } g \in G.
$$
\end{enumerate}
\end{prop}

Recall that, the \emph{core} of a subgroup $A \le B$ in the group $B$ is
the largest normal subgroup of $B$ contained in $A$. We say that $A$ is
\emph{core-free} in $B$ if its core in $B$ is trivial.
We will use the following classification result proved in \cite{KK}.

\begin{theorem}{\rm \cite[Theorem~3.4]{KK}}\label{core-free}
 Let $\M$ be a regular Cayley map for the dihedral group
$D_n$ with $n>4$ such that cyclic subgroup $L(A_n)$ is core-free in
$\Aut(\M)$. Then $n = 2m,$ $m$ is an odd number, and
$\M$ is equivalent to the Cayley map $CM(D_n, a\langle a^2\rangle
\cup  b\langle a^2 \rangle, p)$, where $ p=(b, a, a^2b, a^3, a^4b,
\ldots , a^{n-2}b, a^{n-1}).$
\end{theorem}

\begin{remark}\label{small-maps}
 For small positive integers $n$ with $2 \le n \le 4$,
the regular Cayley maps $\M$ on $D_n$ such that cyclic subgroup
$L(A_n)$ is core-free in $\Aut(\M)$ can be easily obtained (see also
\cite[Theorem~2.8]{KMM}). Up to equivalence, these are the following
maps:
\begin{itemize}
\item $CM(D_2,\{b,ab, a\},(b,ab, a));$
\item $CM(D_3,\{a^{-1},a,b,a^2b\},(a^{-1},a,b,a^2b));$
\item $CM(D_4,\{b,a,a^{-1}\},(b, a, a^{-1}))$.
\end{itemize}
\end{remark}

We will also make use of two further results from \cite{KK}.

\begin{prop}{\rm \cite[Corollary~4.2]{KK}}\label{half-kernel}
Let $\M= CM(D_n,X, p)$ be a regular Cayley map with associated
skew-morphism $\psi$ and power function $\pi$ and let $M = A_n \cap \ker(\pi)$.
If $|M| > 2,$ then $L(M)$ is normal in $\Aut(\M)$.
\end{prop}

\begin{theorem}{\rm \cite[Theorem~4.5]{KK}}\label{kernel-at-least-4}
Let $n \ge 2,$ and $\M = CM(D_n,X,p)$ be a regular Cayley map with
associated power function $\pi$. Then either $n = 3,$ $\M$ is the embedding of the octahedron into the sphere and $|\ker(\pi)| = 2,$ or $|\ker(\pi)| \ge 4$.
\end{theorem}

Finally, we will need the following result about skew-morphisms of cyclic groups.

\begin{prop}{\rm \cite[Corollary~3.3]{KN}}\label{KN}
Let $\varphi$ be an arbitrary skew-morphism of the cyclic group $\ZZ_n$, and $T$ be an orbit
of  $\left< \varphi \right>$  such that $T$ generates $\ZZ_n$.
Then $T$ contains a generator of $\ZZ_n$.
\end{prop}

\section{The reflection index}

As set in the previous section, the dihedral group $D_n$ of order $2n$ is given by the presentation
$D_n=\left< a,b \mid a^n=b^2=1, bab=a^{-1} \right>,$  $A_n=\left< a \right>$
and $B_n = D_n \setminus A_n$.

\begin{lemma}\label{generating-set}
Let $\M=CM(D_n,X,p)$ be a regular Cayley map for the dihedral group $D_n$.
If $X \cap A_n \ne \emptyset$ then $A_n=\left< X \cap A_n \right>$.
\end{lemma}

\begin{proof}
We proceed by induction on $n$. The lemma can be checked readily if $n \le 6,$ hence
from now on it is assumed that $n > 6$.
We may also assume that $n$ is even, because if $n$ is odd, then the lemma follows
from the classification of non-balanced regular Cayley maps for $D_n$ with $n$ odd
given in \cite[Theorem~3.2]{KMM}.

If $L(A_n)$ has trivial core in $\Aut(\M),$ then
the lemma follows from Theorem~\ref{core-free} and Remark~\ref{small-maps}.
Thus we assume that there exists a subgroup $P <  A_n$
such that $L(P) \triangleleft \Aut(\M)$ and $|P|=p$ is a prime number.

Let $X_a = X \cap A_n$ and $A^*=\left< X_a \right>$. Assume, towards
a contradiction, that $A^* \ne A_n$. Since $L(P) \triangleleft
\Aut(\M),$ $D_n/P$ is a block system for $\Aut(\M)$. Let us consider
the quotient map $\M/P=CM(D_n/P, X/P, p^{D_n/P}),$ see
Proposition~\ref{quo} and the preceding paragraph. Since $X/P \cap
A_n/P \ne \emptyset,$ the induction hypothesis gives $\left< X_a/P
\right> = A_n/P$. Thus $A_n=\left< X_a P \right> = \left< X_a
\right> P = A^* P$. Since $A^* \ne A_n,$ $P \not\le A^*,$ and so
$A_n = P \times A^*$ (also, $p \nmid |A^*|$ as $A_n$ is a cyclic
group). A right $P$-coset intersecting $X$ will have the same number
of common elements with $X$. Since $P \not\le A^*,$ this number is
$1$. This implies that $|\psi|=|\psi^{D_n/P}|,$ where $\psi^{D_n/P}$
is the skew-morphisms of $D_n/P$ induced by $\psi$ (see
Proposition~\ref{quo}). Thus $\pi$ and $\pi^{D_n/P}$ take the same
number of values, and therefore,  $|D_n :\ker(\pi)| = |D_n/P :
\ker(\pi^{D_n/P})|,$ which gives $|\ker(\pi)| = p \cdot
|\ker(\pi^{D_n/P})|$. By Theorem~\ref{kernel-at-least-4},
$|\ker(\pi^{D_n/P})| \ge 4$ or $n/p=3$. Since $n$ is even and larger
than $6,$ we find that $\ker(\pi) \ge 4p$. By
Proposition~\ref{half-kernel},  and using also that $p^2 \nmid |A_n|$,
there exists a subgroup $Q < A_n$
such that $L(Q) \triangleleft \Aut(\M)$ and $|Q|=q$ is a prime
number and $p \ne q$. Then $Q \le A^* = \langle X_a \rangle,$
 and hence $A^*=A^* Q = \langle X_aQ \rangle$.
 This implies that $\langle X_a/Q \rangle = A^*/Q$. On
the other hand, by induction, $X_a/Q$ must generate $A_n/Q,$
implying that $A^* = A_n,$ a contradiction. The lemma is proved.
\end{proof}

\begin{lemma}\label{cyclic-generator}
Let $\M=CM(D_n,X,p)$ be a regular Cayley map for the dihedral group $D_n$.
If $X \cap A_n \ne \emptyset$ then there exists $x \in X$ such that
$A_n=\left< x \right>$.
\end{lemma}

\begin{proof}
We set $X_a=X \cap A_n$. Let $N$ be the largest subgroup of $A_n$
 such that the $L(N)$-orbits form a system of blocks
for $\Aut(\M)$. Then $\M/N=CM(D_n/N,X/N,p^{D_n/N})$ is also a
regular Cayley map with associated skew-morphisms $\psi^{D_n/P}$
 by Proposition~\ref{quo}. Clearly, $|A_n/N| \ge 2$.
Observe that $\M/N$ is a core-free map, i.\ e., the cyclic subgroup
$L(A_n/N)$ is core-free in $\Aut(\M/N)$. By  Theorem~\ref{core-free}
and Remark~\ref{small-maps}, we may assume, up to equivalence, that
one of the following holds:
\begin{enumerate}[(i)]
\item $n/|N|=2,$ and $p^{D_n/N}= (N b, N ab, N a);$
\item $n/|N|=3,$ and $p^{D_n/N}= (N a^{-1}, N a, N b, N a^2b);$
\item $n/|N|=4,$ and $p^{D_n/N}= (N b,  N a, N a^{-1});$
\item $n/|N| \equiv 2(\mod 4),$ and
$p^{D_n/N}= (Nb, Na, Na^2b, \ldots, Na^{-1})$.
\end{enumerate}

We give a detailed proof for cases (i) and (ii), and only a sketch
for cases (iii) and (iv) because they are obtained by repeating very
similar  processes.
\medskip

(i): Let $\varphi=\psi^3$.  It is easily seen that $X_a$ is an orbit under
$\langle \varphi \rangle$.  The power function $\pi^{D_n/N}$ takes on $1$ for every element in $D_n/N$. Thus $\pi(x) \equiv 1 (\mod 3)$ for every $x \in D_n,$ see Proposition~\ref{quo}(5), and thus we can write for any $x,y \in D_n,$
$$
\varphi(x y)=\psi^3(xy)=\psi^3(x)\psi^{\pi(x)+\pi(\psi(x))+\pi(\psi^2(x))}(y)=
\psi^3(x)\psi^{3k}(y) = \varphi(x)\varphi^k(y)
$$
for some integer $k$. This gives that $\varphi$ is a skew-morphism of $D_n$.
As $\left< X_a \right>=A_n,$ see Lemma~\ref{generating-set}, $\varphi$ maps
$A_n$ to itself. Let $\bar{\varphi}$ be the restriction of $\varphi$ to $A_n$. Then $\bar{\varphi}$ is a skew-morphism of the cyclic group $A_n$ such that $X_a$ is an
$\left< \bar{\varphi} \right>$-orbit which generates $A_n$.
By Proposition~\ref{KN}, $X_a$ contains a generator of $A_n$.
\medskip

(ii): Let $\varphi=\psi^4$.  In this case $X_a$ splits into two
orbits under $\left< \varphi \right>,$ say $X_a^{(i)},$ $i=1,2$.
Furthermore, they satisfy $X_a^{(1)}=\big( X_a^{(2)} \big)^{-1}$.
Notice that both $X_a^{(1)}$ and $X_a^{(2)}$ generate $A_n$. The
power function $\pi^{D_n/N}$ is given as follows:
$\pi^{D_n/N}(N)=\pi^{D_n/N}(N a^2 b)=1,$ $\pi^{D_n/N}(N
a)=\pi^{D_n/N}(N ab)=2,$ and $\pi^{D_n/N}(N a^2)=\pi^{D_n/N}(N
b)=3$. These imply with Proposition~\ref{quo}(5) that, for every $x
\in D_n$
$$
\pi(x)+\pi(\psi(x))+\pi(\psi^2(x)+\pi(\psi^3(x)) \equiv 0 (\mod 4).
 $$
As in case (i), we deduce from this that $\varphi$ is a skew-morphism of $D_n$
which maps $A_n$ to itself. Now, the restriction of $\varphi$ to $A_n$
is a skew-morphism of $A_n$ with $(X_a)^{(1)}$ is a generating orbit.
The proof can be finished again by Proposition~\ref{KN}.
\medskip

(iii) -- (iv): We let $\varphi=\psi^3$ in case (iii), and $\varphi=\psi^2$ in
case (iv). Then we show that $\varphi$ is a skew-morphism of
$D_n$ which maps $A_n$ to itself. In case (iii) the set $X_a$ splits into two orbits
under $\left< \varphi \right>$ which are inverses to one another, while $X_a$ is an
$\left< \varphi \right>$-orbit in case (iv).
Finally, we finish the proof by applying Proposition~\ref{KN}
to the restriction of $\varphi$ to $A_n$.
\end{proof}

\begin{lemma} \label{refl-groupauto}
Let ${\cal M} = CM(D_n,X,p)$ be a regular Cayley map for the
dihedral group $D_n$. Now ${\cal M}$ is reflexible if and only if there exists an automorphism $\alpha \in \Aut(D_n)$ such that $\alpha(X)=X$ and for any $x \in X$, $\alpha(p(x)) = p^{-1}(\alpha(x))$.
\end{lemma}

\begin{proof} Let $p=(x_0,x_1,\ldots,x_{d-1})$.

Suppose there exists an automorphism $\alpha \in \Aut(D_n)$ such that $\alpha(X)=X$ and for any $x \in X$, $\alpha(p(x)) = p^{-1}(\alpha(x))$.  Then $\alpha(1_{D_n})=1_{D_n}$. For any $g \in D_n$ and $x \in X$, $\alpha(g)^{-1}\alpha(gx) = \alpha(g)^{-1}\alpha(g)\alpha(x)=\alpha(x) \in X$ and
$$
\alpha(g)^{-1}\alpha(gp(x)) =\alpha(g)^{-1}\alpha(g)\alpha(p(x))=
\alpha(p(x))=p^{-1}(\alpha(x)) = p^{-1}( \alpha(g)^{-1}\alpha(gx)).$$
Hence $\alpha$ is an antirotary mapping of ${\cal M}$, and consequently ${\cal M}$ is reflexible by Proposition~\ref{antirotary}.

Assume that ${\cal M}$ is reflexible and balanced.  Note that for
any $i=0,\ldots,d-1$, $x_i \in B_n$. Since ${\cal M}$ is regular, we
can choose  an orientation-reversing map automorphism $\psi$ of
${\cal M}$ satisfying that $\psi(1_{D_n})=1_{D_n}$ and
$\psi(x_i)=x_{-i}$ for any $i=0,\ldots, d-1$. This means that
$\psi(X)=X$ and $\psi(p(x_i))=\psi(x_{i+1}) =
x_{-i-1}=p^{-1}(x_{-i})=p^{-1}(\psi(x_i))$. Suppose that there
exists a positive integer $j$ with $j \ge 1$ such that for any
$x_{i_1}, x_{i_2}, \ldots, x_{i_j}\in X$  (not
necessarily distinct), $\psi(x_{i_1}x_{i_2} \cdots
x_{i_j})=\psi(x_{i_1})\psi(x_{i_2}) \cdots \psi(x_{i_j}) =
x_{-i_1}x_{-i_2} \cdots x_{-i_j}$. Then for any $k=0,\ldots, d-1$,
\begin{eqnarray*} \psi(x_{i_1}x_{i_2} \cdots x_{i_j}x_k) &=& \psi(x_{i_1}x_{i_2} \cdots x_{i_j})x_{c(-i_j)+c(i_j)-k}=x_{-i_1}x_{-i_2} \cdots x_{-i_j}x_{-k} \\ &=& \psi(x_{i_1})\psi(x_{i_2}) \cdots \psi(x_{i_j})\psi(x_k) \end{eqnarray*}
by Proposition~\ref{refle-images}. Therefore $\psi$ is a group automorphism of $D_n$.

For the remaining case, let ${\cal M}$ be reflexible and $X \cap A_n \neq \emptyset$. By Lemma~\ref{cyclic-generator}, there exists $x_k \in X$ such that $A_n = \langle x_k \rangle$. Let $c(k)=r$, namely $x_r = x_k^{-1}$. Since ${\cal M}$ is regular, we can choose  an orientation-reversing map automorphism $\psi_1$ of ${\cal M}$ satisfying that $\psi_1(1_{D_n})=1_{D_n}$ and $\psi_1(x_k)=x_r = x_k^{-1}$. Note that for any $i=0,\ldots, d-1$, $\psi_1(x_{k+i})=x_{r-i}$ and hence $\psi_1(X)=X$ and
$\psi_1(p(x_{k+i}))=\psi_1(x_{k+i+1})= x_{r-i-1} = p^{-1}(x_{r-i}) = p^{-1}(\psi_1(x_{k+i}))$.
 Suppose that for some positive integer $j$ with $j \ge 1$, $\psi_1(x_k^{j}) = \psi_1(x_k)^j
 = x_k^{-j}$. Then $$\psi_1(x_k^{j+1}) = \psi_1(x_k^{j}x_k)=\psi_1(x_k^{j})x_{c(r)+c(k)-k} = x_k^{-j}x_r = x_k^{-j-1} = \psi_1(x_k)^{j+1}$$ by Proposition~\ref{refle-images}.
 Hence for any $a^i \in A_n$, $\psi_1(a^i)=a^{-i}$ and for any $a^ib \in B_n$, $\psi_1(a^ib) \in B_n$. Let $b=x_k^s x_t$ for some integers $s$ and $t=0,\ldots,d-1$. Now $x_t \in B_n$ and
$\psi_1(b)=\psi_1(x_k^s x_t)=\psi_1(x_k^s)x_{c(r)+c(k)-t}=x_k^{-s}x_{c(r)+c(k)-t}$. Therefore for any $i=0,\ldots,n-1$,
$$\psi_1(x_k^ib) =\psi_1(x_k^{i+s}x_t) = x_k^{-i-s}x_{c(r)+c(k)-t}= x_k^{-i}x_k^{-s}x_{c(r)+c(k)-t}=\psi_1(x_k^i)\psi_1(b).$$
Therefore $\psi_1$ is a group automorphism of $D_n$.
\end{proof}

The proof of Lemma~\ref{refl-groupauto} implies that if ${\cal M} = CM(D_n,X,p)$ is a reflexible regular balanced Cayley map then any orientation-reversing automorphism of ${\cal M}$ fixing $1_{D_n}$ is also a group automorphism. For  a reflexible regular non-balanced Cayley map ${\cal M} = CM(D_n,X,p)$, the proof of Lemma~\ref{refl-groupauto} implies that there exists a $\alpha_{-1,j} \in \Aut(D_n)$ which is also an orientation-reversing map automorphism  of ${\cal M}$. So we have the following corollay.

\begin{cor} \label{non-bal-ref-auto}
Let ${\cal M} = CM(D_n,X,p)$ be a regular non-balanced Cayley map for the
dihedral group $D_n$. Now ${\cal M}$ is reflexible if and only if there exists an automorphism $\alpha \in \Aut(D_n)$ such that $\alpha(a)=a^{-1}, \ \alpha(X)=X$ and for any $x \in X$, $\alpha(p(x)) = p^{-1}(\alpha(x))$.
\end{cor}

 We call  such $\alpha$  in Corollay~\ref{non-bal-ref-auto} a \emph{partially inverting reflection} of ${\cal M}$. In fact, there is a unique partially inverting reflection of ${\cal M}$ and for any $x_i, x_k \in X \cap A_n$, $i +c(i) \equiv k + c(k) (\mod d)$.
 Let $r = \min\{ k \in \mathbb{Z} \ \mid \ p^k(x_i) = x_i^{-1}, \ x_i \in X \cap A_n, \ k \ge 1 \}$ and call it the \emph{reflection index} of ${\cal M}$ and denote it by $r({\cal M})$.

\section{Proof of Theorem~\ref{main-result}}

In this Section we prove Theorem~\ref{main-result}, namely we classify reflexible regular Cayley maps for the dihedral groups $D_n$. All lemmas in this section give a proof
of Theorem~\ref{main-result}.

If ${\cal M} = CM(D_n,X,p)$ is a reflexible regular Cayley map for the dihedral group $D_n$ and its valency $d$ is 2, then one can easily show that ${\cal M}$ is isomorphic to $CM(D_n, \{b,ab\},$ $ p=(b, ab))$, which is reflexible regular.  For $d=3$, we have the following lemma.

\begin{lemma} \label{d=3}
Let  ${\cal M} = CM(D_n, X, p)$ be a  Cayley map for the dihedral
group $D_n$ of valency $d=3$. Now ${\cal M}$ is reflexible regular
if and only if ${\cal M}$ is isomorphic to
$CM(D_2,\{b,ab,a\},(b,ab,a))$, $CM(D_3,\{b,ab,a^2b\},(b,ab,a^2b))$
or $CM(D_4,\{b, a,a^{-1} \},(b,$ $a,a^{-1}))$.
\end{lemma}

\begin{proof}
The underlying graphs of
$CM(D_2,\{b,ab,a\},(b,ab,a))$, $CM(D_3,\{b,ab,a^2b\},(b,$
$ab,a^2b))$ and $CM(D_4,\{b, a,$ $a^{-1} \},(b,a,a^{-1}))$ are
 the complete graph $K_4$, the complete bipartite graph
$K_{3,3}$ and $3$-cube $Q_3$, respectively, and  one can easily show
that these Cayley maps are reflexible regular.

Let ${\cal M} = CM(D_n, X, p)$ be a reflexible regular  Cayley map
for the dihedral group $D_n$ of valency $d=3$.  If
$n=2$, then one can easily show that ${\cal M} =
CM(D_2,\{b,ab,a\},(b,$ $ab,a))$. Assume that $n \ge 3$. Suppose
that ${\cal M}$ is balanced. Then ${\cal M}$ is isomorphic to
$CM(D_n,\{b,ab,a^{\ell+1}b\},$ $(b,ab,a^{\ell+1}b))$ for some $\ell$
satisfying $\ell^2+\ell+1 \equiv 0(\mod n)$. By
Lemma~\ref{refl-groupauto}, one can assume that there exists a group
automorphism $\psi$ of $D_n$ such that $\psi(b)=b, \
\psi(ab)=a^{\ell+1}b$ and $\psi(a^{\ell+1}b)=ab$. Now we have
$\psi(a)=\psi(abb)=a^{\ell+1}$. So $ab = \psi(a^{\ell+1}b) =
a^{(\ell+1)^2}b=a^{\ell^2+2\ell+1}b =a^{\ell}b$. This means that
$\ell=1$ and hence ${\cal M}$ is isomorphic to
$CM(D_3,\{b,ab,a^2b\},(b,ab,a^2b))$.

For the remaining case, let $X \cap A_n \neq \emptyset$ and let $p=(x_0,x_1,x_2)$ such that $x_0 \in B_n$ and $x_1^{-1} = x_2$. By Proposition~\ref{iso} and
Lemma~\ref{cyclic-generator},
one can assume that $p=(b,a,a^{-1})$. Now, $\ker(\pi)=\langle a^2, ab \rangle$ and for any $g \in D_n \setminus \ker(\pi)$, $\pi(g)=2$. This means that
$$\varphi(ab)=\psi(a)\varphi^2(b)=a^{-2} \ \ \mbox{and} \ \ \varphi(ba^{-1}) = \varphi(b)\varphi^2(a^{-1}) = a^2.$$  Since $ab=ba^{-1}$, we have $a^2=a^{-2}$.
 Since we assume $n \ge 3$,  $n=4$ and ${\cal M}$ is
isomorphic to $CM(D_4,\{b, a,a^{-1} \},(b,a,a^{-1}))$.
\end{proof}

In Lemma~\ref{d=3}, we classify reflexible regular Cayley maps for $D_n$ of valency $3$. From now on, we assume that the valency of ${\cal M}$ is greater than or equal to $4$. We  show that the Cayley map ${\cal M}_i$ in Theorem~\ref{main-result} is reflexible regular  and  any reflexible regular Cayley map for $D_n$ with valency $d \ge 4$ is isomorphic to  ${\cal M}_i$ in Theorem~\ref{main-result}.

\begin{lemma} \label{list-reflexible}
For any $i=1,\ldots,6$, the Cayley map ${\cal M}_i$ in Theorem~\ref{main-result} is reflexible regular.
\end{lemma}

\begin{proof} If $n$ is even then let $n=2m$. For any $i=1,\ldots,6$,  let $\varphi_i, \pi_i$ and $\psi_i$ be defined as follows. Now one can check that $\varphi_i$ is a skew-morphism of
$D_n$ with associated power function $\pi_i$ and $X_i$ is an orbit
of $\varphi_i$ which is a generating set of $D_n$ and closed under
inverse, and the restriction of $\varphi_i$ on $X_i$ is $p_i$. Hence
${\cal M}_i$ is regular. Furthermore $\psi_i$ is a group
automorphism of $D_n$ satisfying $\psi_i(X_i)=X_i$ and
$\psi_i(p(x))=p^{-1}(\psi_i(x))$ for any $x \in X_i$.   So ${\cal
M}_i$ is reflexible  by Lemma~\ref{refl-groupauto}. Let $j$ be an
integer.

\noindent (1) \vspace{-.5cm} \begin{eqnarray*} \varphi_1(a^j)&=& a^{j\ell} \ \ \mbox{and} \ \ \varphi_1(a^jb)=a^{j\ell+1}b  \\
\pi_1(g)&=&1  \ \ \mbox{for any $g \in D_n$} \\
\psi_1(a^j)&=& a^{-j\ell} \ \ \mbox{and} \ \ \psi_1(a^jb)=a^{-j\ell}b.  \end{eqnarray*}
\noindent (2)\vspace{-.5cm} \begin{eqnarray*}
\varphi_2(a^j) &=& \begin{cases} a^{-j} & \mbox{ if $j$ is even} \\  a^{j+1}b & \mbox{ if $j$ is odd} \end{cases} \quad
\varphi_2(a^jb) = \begin{cases} a^{j+1} & \mbox{ if $j$ is even} \\  a^{-j}b & \mbox{ if $j$ is odd} \end{cases}
\\
\pi_2(a^{2j+1}) &=& \pi_2(a^{n-2j-2}b) \equiv 4j+3 (\text{mod } n) \ \ \mbox{and} \\ \pi_2(a^{2j}) &=& \pi_2(a^{n-2j-1}b) \equiv 4j+1 (\text{mod } n) \\
\psi_2(a^j)&=&a^{-j} \ \ \mbox{and} \ \ \psi_2(a^jb)=a^{-j}b. \end{eqnarray*}

\noindent (3) \vspace{-.5cm} \begin{eqnarray*}
\varphi_3(a^j) &=& \begin{cases} a^{\frac{j}{2}m-j} & \mbox{ if $j$ is even} \\  a^{j+1+\frac{j+1}{2}m}b & \mbox{ if $j$ is odd} \end{cases} \quad
\varphi_3(a^jb) = \begin{cases}
a^{j+1+\frac{j}{2}m} & \mbox{ if $j$ is even} \\
a^{\frac{j+1}{2}m-j}b & \mbox{  if $j$ is odd}
\end{cases}
\\
\pi_3(a^{2j+1}) &=& \pi_3(a^{n-2j-2}b) \equiv 4j+3 (\text{mod } n) \ \ \mbox{and} \\ \pi_3(a^{2j}) &=& \pi_3(a^{n-2j-1}b) \equiv 4j+1 (\text{mod } n) \\
\psi_3(a^j)&=&a^{-j} \ \ \mbox{and} \ \ \psi_3(a^jb)=a^{-j}b. \end{eqnarray*}

\noindent (4) \vspace{-.5cm} \begin{eqnarray*}
\varphi_4(a^j) &=& \begin{cases} a^{-j} & \mbox{ if $j \equiv 0 (\mod 3)$} \\  a^{-j+1}b & \mbox{ if $j \equiv 1 (\mod 3)$} \\ a^{-j} & \mbox{ if $j \equiv 2 (\mod 3)$} \end{cases} \quad
\varphi_4(a^jb) = \begin{cases} a^{-j+2}b & \mbox{ if $j \equiv 0 (\mod 3)$} \\  a^{-j+2}b & \mbox{ if $j \equiv 1 (\mod 3)$} \\ a^{-j+1} & \mbox{ if $j \equiv 2 (\mod 3)$} \end{cases}
\\
\pi_4(a^{3j}) &=& \pi_4(a^{3j}b) = 1, \  \ \pi_4(a^{3j+1}) = \pi_4(a^{3j+2}b) = 2 \ \ \mbox{and} \ \ \pi_4(a^{3j+2}) = \pi_4(a^{3j+1}b) = 3  \\
\psi_4(a^j)&=&a^{-j} \ \ \mbox{and} \ \ \psi_4(a^jb)=a^{-j+2}b. \end{eqnarray*}

\noindent (5) \vspace{-.5cm} \begin{eqnarray*}
\varphi_5(a^j) &=& \begin{cases} a^{-j} & \mbox{ if $j \equiv 0 (\mod 4)$} \\  a^{-j+1}b & \mbox{ if $j \equiv 1 (\mod 4)$} \\ a^{-j+1+m}b & \mbox{ if $j \equiv 2 (\mod 4)$} \\ a^{-j} & \mbox{ if $j \equiv 3 (\mod 4)$} \end{cases} \quad
\varphi_5(a^jb) = \begin{cases} a^{-j+1+m}b & \mbox{ if $j \equiv 0 (\mod 4)$} \\  a^{-j+2+m}b & \mbox{ if $j \equiv 1 (\mod 4)$} \\ a^{-j+2+m}b & \mbox{ if $j \equiv 2 (\mod 4)$} \\ a^{-j+1} & \mbox{ if $j \equiv 3 (\mod 4)$} \end{cases}
\\
 \pi_5(a^{4j}) &=& \pi_5(a^{4j+1}b) = 1, \  \ \pi_5(a^{4j+1}) = \pi_5(a^{4j}b) = 2 \\ \pi_5(a^{4j+2}) &=& \pi_5(a^{4j+3}b) = 4 \ \ \mbox{and} \ \ \pi_5(a^{4j+3}) = \pi_5(a^{4j+2}b) = 5 \\
\psi_5(a^j)&=&a^{-j} \ \ \mbox{and} \ \ \psi_5(a^jb)=a^{-j+2+m}b. \end{eqnarray*}

\noindent (6) \vspace{-.5cm} \begin{eqnarray*}
\varphi_6(a^j) &=& \begin{cases} a^{-j} & \mbox{ if $j \equiv 0 (\mod 2)$} \\  a^{-j-1}b & \mbox{ if $j \equiv 1 (\mod 2)$} \end{cases} \quad
\varphi_6(a^jb) = \begin{cases} a^{-j-2+m}b & \mbox{ if $j \equiv 0 (\mod 2)$} \\  a^{-j-1+m} & \mbox{ if $j \equiv 1 (\mod 2)$}  \end{cases}
\\
 \pi_6(a^{4j})&=& \pi_6(a^{2j}b) = 1 \ \ \mbox{and} \ \ \pi_6(a^{2j+1}) = \pi_6(a^{2j+1}b) = 4 \\
\psi_6(a^j)&=&a^{-j} \ \ \mbox{and} \ \ \psi_6(a^jb)=a^{-j-2+m}b. \end{eqnarray*}
\end{proof}

\begin{lemma} \label{balanced}
If  ${\cal M} = CM(D_n, X, p)$ is a reflexible regular balanced  Cayley map
for the dihedral group $D_n$, then ${\cal M}$ is isomorphic to  ${\cal M}_1$ in
Theorem~\ref{main-result}.
\end{lemma}

\begin{proof}
Let $p = (b, ab, a^{\ell+1}b, \ldots, a^{\ell^{d-2}+\ell^{d-3}+\cdots+1}b)$ for some positive integer $\ell$ such that $\ell^{d-1}+\ell^{d-2}+\cdots+1 \equiv 0 (\mod n)$.
By Lemma~\ref{refl-groupauto}, one can assume that there exists a group automorphism $\psi$ of $D_n$ such that $\psi(b)=b, \ \psi(ab)=a^{\ell^{d-2}+\ell^{d-3}+\cdots+1}b$ and $\psi(a^{\ell^{d-2}+\ell^{d-3}+\cdots+1}b)=ab$. Let $k=\ell^{d-2}+\ell^{d-3}+\cdots+1$.  Now we have $\psi(a)=\psi(abb)=a^{k}$ and $ab = \psi(a^{k}b) = a^{k^2}b $. So $k^2 \equiv 1 (\mod n)$. Since $\ell^{d-1}+\ell^{d-2}+\cdots+1 = \ell k + 1 \equiv 0 (\mod n)$, we have $k^2 \equiv 1 \equiv -\ell k (\mod n)$.
This implies that $k \equiv -\ell (\mod n)$ and hence $k^2 \equiv \ell^2 \equiv 1 (\mod n)$. Therefore ${\cal M}$ is isomorphic to ${\cal M}_1$ in Theorem~\ref{list-reflexible}.
\end{proof}

From now on, we aim to consider reflexible regular
non-balanced Cayley maps for the dihedral group $D_n$ in terms of
rotation index.

\begin{lemma} \label{r=1}
If  ${\cal M} = CM(D_n, X, p)$ is a reflexible regular Cayley map
for the dihedral group $D_n$ whose rotation index $r({\cal M})$ is $1$, then ${\cal M}$ is isomorphic to  ${\cal M}_4$ or ${\cal M}_5$ in Theorem~\ref{main-result}.
\end{lemma}

\begin{proof}
Let ${\cal M} = CM(D_n, X, p)$ be a reflexible  regular
Cayley map for $D_n$ with $p=(x_0,x_1, \ldots, x_{d-1})$ satisfying
$x_0^{-1} = x_1$.

For the first consideration, let $d =4$. By Proposition~\ref{iso} and
Lemma~\ref{cyclic-generator}, we can assume that
$p=(a^{-1},a, b, a^kb)$ for some $k$.  Now $\pi(a)=\pi(a^kb)=2, \ \pi(b)=1$ and $\pi(a^{-1})=3$.
So $\varphi(ab)=\varphi(a)\varphi^2(b) = ba^{-1}=ab$ and $\varphi(ba^{-1}) = \varphi(b)\varphi(a^{-1})=a^kba=a^{k-1}b$. Since $ab=ba^{-1}$, we have $k=2$. The covalency (face size) of ${\cal M}$ is the order $n$ of $a$, which equals to $3$ times the order of $a^{-1}ba^2b=a^{-3}$. This implies that $n$ is a multiple of $3$ and hence ${\cal M}$ is isomorphic to  ${\cal M}_4$.

Assume that $d \ge 5$. Note that if $x_i \in A_n$ then $x_i^{-1} = x_{d+1-i}$, and  if $x_i \in B_n$ then $x_{d+1-i}$ also belongs to $B_n$.  Our discussion can be divided into the following cases.
\medskip

\noindent{\it Case 1.} Suppose that $x_{d-1}, x_2 \in A_n$, namely $x_{d-1}^{-1} = x_2$. Let $x_1=a^i$ and $x_2=a^j$. Now $x_{d-1}=a^{-j}$ and $x_0=a^{-i}$. This implies that $\pi(a^{-j})=\pi(a^{-i})=\pi(a^i)=-1$. Hence
$$\varphi(a^{-j-i}) =\varphi(a^{-j})\varphi^{-1}(a^{-i})=a^{-i-j} \ \ \mbox{and} \ \  \varphi(a^{-i-j}) = \varphi(a^{-i})\varphi^{-1}(a^{-j})=a^i \varphi^{-1}(a^{-j}).
$$
So $\varphi^{-1}(a^{-j}) = x_{d-2} = a^{-2i-j}$, which implies that $x_3 = x_{d-2}^{-1} = a^{2i+j}$ and $\pi(a^{-2i-j})=-1$. Now we have
\begin{eqnarray*} \varphi(x_{d-2}x_{d-1})&=& \varphi(a^{(-2i-j)-j}) =\varphi(a^{-2i-j})\varphi^{-1}(a^{-j})=a^{-j+(-2i-j)} \ \ \mbox{and} \\  \varphi(a^{-j+(-2i-j)}) &=&\varphi(a^{-j})\varphi^{-1}(a^{-2i-j})=a^{-i} \varphi^{-1}(a^{-2i-j})\end{eqnarray*}
 and hence $\varphi^{-1}(a^{-2i-j}) = x_{d-3} = a^{-i-2j}$. By continuing the similar process, one can get that for any $x_k \in X$, $x_k \in A_n$. It contradicts that $X$ is a generating set of $D_n$.
\medskip

\noindent{\it Case 2.} Suppose that $x_{d-2}, x_{d-1}, x_2, x_3 \in B_n$. Then $\pi(x_{d-2})=\pi(x_2)=1$, $\pi(x_{d-1})=\pi(x_1)=2$ and $\pi(x_0)=-1$. By Proposition~\ref{iso}, one can assume that $x_{d-1}=b$. Let $x_{d-2}=a^{k_1}b$ and $x_2 = a^{k_2}b$. Now
$$ \varphi(a^{k_1}ba^{-i}) =\varphi(a^{k_1}b)\varphi(a^{-i})=ba^i \ \ \mbox{and} \ \   \varphi(a^ia^{k_1}b) = \varphi(a^i)\varphi^2(a^{k_1}b)=a^{k_2}ba^{-i},
$$
 which implies that  $k_2 = -2i$. Since
\begin{eqnarray*}  \varphi(a^{k_1}ba^i) &=& \varphi(a^{k_1}b)\varphi(a^i)=ba^{k_2}b = a^{-k_2}=a^{2i} \ \ \mbox{and} \\   \varphi(a^{-i}a^{k_1}b) &=& \varphi(a^{-i})\varphi^{-1}(a^{k_1}b)=a^i\varphi^{-1}(a^{k_1}b),
\end{eqnarray*}
we have $\varphi^{-1}(a^{k_1}b) = x_{d-3} = a^i=x_1$, which is a contradiction.
\medskip

\noindent{\it Case 3.} Suppose that $x_{d-1},x_2 \in B_n$ and $x_{d-2}, x_3 \in A_n$. Then $\pi(x_{d-2})=\pi(x_2)=-4$, $\pi(x_{d-1})=\pi(x_1)=2$ and $\pi(x_0)=-1$.  Let $x_{d-1}=a^{k_1}b, \ x_2 = a^{k_2}b$ and $x_3 = x_{d-2}^{-1}=a^j$. Now
$$\varphi(a^{-i}a^{k_1}b) =\varphi(a^{-i})\varphi^{-1}(a^{k_1}b)=a^ia^{-j} \ \ \mbox{and} \ \   \varphi(a^{k_1}ba^i) = \varphi(a^{k_1}b)\varphi^2(a^i)=a^{-i}a^j,
$$
 which implies that  $2i=2j$. Furthermore we have
$$\varphi(a^{-2i}) =\varphi(a^{-i})\varphi^{-1}(a^{-i})=a^ia^{k_1}b \ \ \mbox{and} \ \
 \varphi(a^{-2j}) =\varphi(a^{-j})\varphi^{-4}(a^{-j})=a^{k_1}b \varphi^{-4}(a^{-j}).
$$
Since $a^{-2i}=a^{-2j}$, we have $\varphi^{-4}(a^{-j})= x_{d-6} =  a^{-i}=x_0$. This implies that $d=6$ and $p=(a^{-i}, a^i ,  a^{k_2}b, a^j, a^{-j}, a^{k_1}b)$. Since $2i=2j$, $A_n =\langle a^i \rangle = \langle a^j \rangle$ be Lemma~\ref{cyclic-generator}. By Proposition~\ref{iso},  we can assume that $p=(a^{-1}, a ,  b, a^j, a^{-j}, a^{k_1}b)$ up to isomorphism. Since $2j=2i=2$, $n$ is even and $j = 1+\frac{n}{2}$. Note that $\pi(a^{1+\frac{n}{2}})=-1$. Since
$$\varphi(a^{1+(1 +\frac{n}{2})}) =\varphi(a)\varphi^{2}(a^{1 +\frac{n}{2}})=ba^{k_1}b \ \ \mbox{and} \ \
\varphi(a^{(1 +\frac{n}{2})+1}) =\varphi(a^{1 +\frac{n}{2}})\varphi^{-1}(a)=a^{-1+\frac{n}{2}}a^{-1} ,$$
$k_1 = 2+ \frac{n}{2}$. The covalency (face size) of ${\cal M}$ is the order $n$ of $a$, which equals to $4$ times the order of $a^{-1}\cdot b\cdot a^{1+\frac{n}{2}}\cdot a^{2+\frac{n}{2}}  b=a^{-4}$. This implies that $n$ is a multiple of $4$.   Note that $\pi(a)=2$ and by Proposition~\ref{properties-skew}, \begin{eqnarray*} \pi(a^2)&=& \sum_{i=0}^{\pi(a)-1} \pi(\varphi^i(a))= \pi(a)+\pi(b) = 4,  \\
\pi(a^3)&=& \sum_{i=0}^{\pi(a^2)-1} \pi(\varphi^i(a))= \pi(a)+\pi(b)+\pi(a^{ 1+\frac{n}{2}})+\pi(a^{ -1+\frac{n}{2}}) = 5 \ \ \mbox{and}  \\
 \pi(a^4)&=& \sum_{i=0}^{\pi( a^3)-1} \pi(\varphi^i(a))= \pi(a)+\pi(b)+\pi(a^{ 1+\frac{n}{2}})+\pi(a^{ -1+\frac{n}{2}}) + \pi(a^{2+ \frac{n}{2}}b)= 1. \end{eqnarray*}
Hence $\langle a^4 \rangle$ is the cyclic subgroup of $\ker(\pi)$ of
the maximum order.  Since $\pi(a) = \pi(a^{ -1+\frac{n}{2}})$, both
$a$ and $a^{ -1+\frac{n}{2}}$ belong to the same right coset of
$\ker(\pi)$, namely
  $a^{-2+\frac{n}{2}} \in \langle a^4 \rangle $. This implies that $\frac{n}{2} \equiv 2 (\mod 4)$,  namely $n \equiv 4 (\mod 8)$, and hence ${\cal M}$ is isomorphic to  ${\cal M}_5$.
\end{proof}

\begin{lemma} \label{r=2}
If  ${\cal M} = CM(D_n, X, p)$ is a reflexible regular  Cayley map
for the dihedral group $D_n$ whose rotation index $r({\cal M})$ is $2$, then ${\cal M}$ is isomorphic to  ${\cal M}_2$ or ${\cal M}_3$ in Theorem~\ref{main-result}.
\end{lemma}

\begin{proof}
Let ${\cal M} = CM(D_n, X, p)$ be a reflexible  regular
Cayley map for $D_n$  with $p=(x_0,x_1, \ldots, x_{d-1})$
satisfying $x_0 \in B_n$ and $x_{d-1}^{-1} = x_1$.
 By Proposition~\ref{iso}, we can assume that $x_0 =b$. Let $x_1 =x_{d-1}^{-1} =a^i$.
 Since ${\cal M}$ is regular and reflexible, there exists an orientation-reversing map automorphism $\psi$ of $D_n$ such that $\psi(x_k)=x_{-k}$ for any $x_k \in X$. By Corollay~\ref{non-bal-ref-auto}, $\psi$ is also a group automorphism. Note that $\psi(a)=a^{-1}$ and $\psi(b)=b$. This implies that for any $x_k \in X$, if $x_k = a^j$ then $x_{-k} = a^{-j}$, and if $x_k = a^jb$ then $x_{-k} = a^{-j}b$.

  Suppose that $x_{d-2}^{-1}=x_2$. Let $x_2 = x_{d-2}^{-1}=a^j$. Now $\pi(a^{-j})=\pi(a^{-i}) = \pi(b)=\pi(a^i)=-1$ and hence
  $$\varphi(a^ib) = \varphi(a^i)\varphi^{-1}(b)=a^ja^{-i}=a^{j-i} \ \ \mbox{and} \ \  \varphi(ba^{-i}) = \varphi(b)\varphi^{-1}(a^{-i})=a^{i-j}.$$
 So $2i \equiv 2j (\mod n)$. Since $a^{-2i}=a^{-2j}$ and
 $$\varphi(a^{-2i}) = \varphi(a^{-i})\varphi^{-1}(a^{-i})=ba^{-j} \ \ \mbox{and} \ \  \varphi(a^{-2j}) = \varphi(a^{-j})\varphi^{-1}(a^{-j})=a^{-i}\varphi^{-1}(a^{-j}),$$
 we have $\varphi^{-1}(a^{-j})=a^{i+j}b$. Also
  $$\varphi(a^{-j+i}) = \varphi(a^{-j})\varphi^{-1}(a^i)=a^{-i}b \ \ \mbox{and} \ \  \varphi(a^{i-j}) = \varphi(a^{i})\varphi^{-1}(a^{-j})=a^j\varphi^{-1}(a^{-j}),$$
 which means that $\varphi^{-1}(a^{-j})=a^{-i-j}b$. So we have $2j \equiv 2i \equiv -2j (\mod n)$, and hence $4j \equiv 0 (\mod n)$. This implies that $\{a^{-i}, a^i \} = \{a^{-j}, a^j \} =\{a^{\frac{n}{4}}, a^{\frac{3n}{4}} \}$, a contradiction.

 Now we can assume that $x_{d-2}, x_2 \in B_n$. Let $x_{d-2} = a^{-k_1}b$ and $x_2 = a^{k_1}b$. Note that $\pi(a^{-i})=\pi(b)=-1$ and $\pi(a^{-k_1}b)=\pi(a^i)=3$. Since
 $$\varphi(a^{-k_1}bb)=\varphi(a^{-k_1}b)\varphi^3(b)=a^{-i}\varphi(a^{k_1}b) \ \ \mbox{and} \ \ \varphi(ba^{k_1}b)=\varphi(b)\varphi^{-1}(a^{k_1}b)=a^{2i},$$
 we have $\varphi(a^{k_1}b)=a^{3i}$, namely $x_3=a^{3i}$ and
 $x_{d-3}=a^{-3i}$.  Hence we have
 $$\varphi(a^ib)=\varphi(a^i)\varphi^3(b)=a^{k_1}ba^{3i} \ \ \mbox{and} \ \ \varphi(ba^{-i})=\varphi(b)\varphi^{-1}(a^{-i})=a^ia^{-k_1}b.$$
 This implies that $2k_1 = 4i$, namely $k_1 = 2i$ or $2i+\frac{n}{2}$ with even $n$. Now our discussion can be divided into the following two
 cases.
\medskip

\noindent{\it Case 1.} Suppose that $k_1 = 2i$. Assume that there exists a positive integer $k$ with $k \ge 1$ such that for non-negative integer $j$ with $j \le k$, $x_{-2j} = a^{-2ji}b, x_{2j}=a^{2ji}b$ and $x_{-2j-1} = a^{(-2j-1)i}, x_{2j+1}=a^{(2j+1)i}$. Then we have $\pi(a^{-2ji}b)=\pi(a^{(2j-1)i})=4j-1$ and $\pi(a^{2ji}b)=\pi(a^{(-2j-1)i})=-4j-1$, and so
\begin{eqnarray*}
\varphi(a^i a^{(2k-1)i}) &=& \varphi(a^i)\varphi^3(a^{(2k-1)i})= a^{2i}b x_{2k+2} \ \ \mbox{and} \\ \varphi(a^i a^{(2k-1)i}) &=& \varphi(ba^{(-2k)i}b) = \varphi(b)\varphi^{-1}(a^{(-2k)i}b)= a^ia^{(-2k-1)i}
\end{eqnarray*}
Now we have $x_{2k+2} = a^{(2k+2)i}b$, and hence $x_{-2k-2}= a^{(-2k-2)i}b$. Since
\begin{eqnarray*}
\varphi(a^i a^{2ki}b) &=& \varphi(a^i)\varphi^3(a^{2ki}b)= a^{2i}bx_{2k+3} \ \  \mbox{and} \\  \varphi(a^i a^{2ki}b)&=& \varphi(ba^{(-2k-1)i}) = \varphi(b)\varphi^{-1}(a^{(-2k-1)i})= a^ia^{(-2k-2)i}b,
\end{eqnarray*}
 we have  $x_{2k+3}= a^{(2k+3)i}$ and $x_{-2k-3} = a^{(-2k-3)i}$.  By induction, for any integer $j$, $x_{2j} = a^{2ji}b$ and $x_{2j+1} = a^{(2j+1)i}$. For ${\cal M}$ to be well-defined, $n$ should be even. Since $X$ is a generating set of $D_n$, the greatest common divisor of $i$ and $n$ is 1. By Proposition~\ref{iso}, we can assume that $i=1$ up to isomorphism, and hence ${\cal M}$ is isomorphic to  ${\cal M}_2$.
\medskip

\noindent{\it Case 2.} Suppose that $k_1 = 2i+\frac{n}{2}$ with even
$n$. Let $m= \frac{n}{2}$. Assume that there exists a positive
integer $k$ with $k \ge 1$ such that for non-negative integer $j$
with $j \le k$, $x_{-2j} = a^{-2ji+jm}b, x_{2j}=a^{2ji+jm}b$ and
$x_{-2j-1} = a^{(-2j-1)i}, x_{2j+1}=a^{(2j+1)i}$. Then we have
$\pi(a^{-2ji+jm}b)=\pi(a^{(2j-1)i})=4j-1$ and
$\pi(a^{2ji+jm}b)=\pi(a^{(-2j-1)i})=-4j-1$, and so
\begin{eqnarray*}
\varphi(a^i a^{(2k-1)i}) &=& \varphi(a^i)\varphi^3(a^{(2k-1)i})= a^{2i+m}b x_{2k+2} \ \ \mbox{and} \\ \varphi(a^i a^{(2k-1)i}) &=& \varphi(a^{-i} a^{(2k+1)i}) = \varphi(a^{-i})\varphi^{-1}(a^{(2k+1)i})= ba^{2ki+km}b,
\end{eqnarray*}
which implies that $x_{2k+2} = a^{(2k+2)i+(k+1)m}b$ and  $x_{-2k-2}= a^{(-2k-2)i-(k+1)m}b$. Since
\begin{eqnarray*}
\varphi(a^{-2i+m}b a^{2ki+km}b) &=& \varphi(a^{-2i+m}b)\varphi^3(a^{2ki+km}b)= a^{-i}x_{2k+3} \ \  \mbox{and} \\  \varphi(a^{-2i+m}b a^{2ki+km}b)&=& \varphi(ba^{(2k+2)i+(k+1)m}b) = \varphi(b)\varphi^{-1}(a^{(2k+2)i+(k+1)m}b)= a^ia^{(2k+1)i},
\end{eqnarray*}
 we have  $x_{2k+3}= a^{(2k+3)i}$ and  $x_{-2k-3} = a^{(-2k-3)i}$.  By induction, for any integer $j$, $x_{2j} = a^{2ji+jm}b$ and $x_{2j+1} = a^{(2j+1)i}$.
 Since $X$ is a generating set of $D_n$, the greatest common divisor of $i$ and $n$ is 1. By Proposition~\ref{iso}, we can assume that $i=1$ up to isomorphism.
  Note that the valency $d$ of ${\cal M}$ is $n$. For ${\cal M}$ to be well-defined, $n$ should be a multiple of $4$. If $n \equiv 4 (\mod 8)$, then $x_m = a^{m+\frac{m}{2}m}b=a^{2m}b=b$,
  a contradiction. Therefore $n$ is a multiple of $8$ and hence ${\cal M}$ is isomorphic to  ${\cal M}_3$.
\end{proof}

\begin{lemma} \label{r=3}
If  ${\cal M} = CM(D_n, X, p)$ is a reflexible regular  Cayley map
for the dihedral group $D_n$ whose rotation index $r({\cal M})$ is $3$, then ${\cal M}$ is isomorphic to  ${\cal M}_6$ in Theorem~\ref{main-result}.
\end{lemma}

\begin{proof}
Let ${\cal M} = CM(D_n, X, p)$ be a reflexible  regular
Cayley map for $D_n$ with $p=(x_0,x_1, \ldots, x_{d-1})$
 and $r({\cal M})=3$. By Proposition~\ref{iso}, let
$x_0=b, \ x_1=a^kb$ and $x_2=x_{d-1}^{-1} = a^i$.  Now
$\pi(a^{-i})=\pi(a^kb)=-2$ and $\pi(b)=1$.

 Suppose that $x_{d-2}, x_3 \in A_n$, namely $x_{d-2}^{-1} = x_3$. Let $x_3=x_{d-2}^{-1} = a^j$. Now $\pi(a^{-j})=\pi(a^i) = -1$ and we have
\begin{eqnarray*} \varphi(a^{-i}b)&=& \varphi(a^{-i})\varphi^{-2}(b)=ba^{-j} \ \ \mbox{and}  \ \   \varphi(a^{-i}b) = \varphi(ba^i)=  \varphi(b)\varphi(a^i)=a^kb a^j \\
 \varphi(a^ib)&=& \varphi(a^i)\varphi^{-1}(b)=a^ja^{-i} \ \ \mbox{and}  \ \   \varphi(a^ib) = \varphi(ba^{-i})=  \varphi(b)\varphi(a^{-i})=a^kb b=a^k.
\end{eqnarray*}
So it holds that $k=2j$ and $k=j-i$. This implies that $j=-i$, a contradiction.

We can assume that $x_{d-2}, x_3 \in B_n$. Let $x_{d-2}=a^{k_1}b$ and
$x_3=a^{k_2}b$. Now $\pi(a^{k_1}b)=\pi(a^i)=4$ and we have
\begin{eqnarray*} \varphi(a^{-i}b)&=& \varphi(a^{-i})\varphi^{-2}(b)=ba^{k_1}b =a^{-k_1}\ \ \mbox{and}  \ \   \varphi(ba^i)=  \varphi(b)\varphi(a^i)=a^kb a^{k_2}b=a^{k-k_2} \\
 \varphi(a^ib)&=& \varphi(a^i)\varphi^{4}(b)=a^{k_2}b\varphi(a^{k_2}b) \ \ \mbox{and}  \ \   \varphi(ba^{-i})=  \varphi(b)\varphi(a^{-i})=a^kb b=a^k.
\end{eqnarray*}
So $k_1=k_2-k$ and $\varphi(a^{k_2}b)=a^{k_2 -k}b=a^{k_1}b$. This
implies that $d=6$ and $p = (b, a^kb, a^i, a^{k_2}b, a^{k_1}b,
a^{-i})$. By Lemma~\ref{cyclic-generator}, $\langle a^i \rangle =
A_n$ and hence the greatest common divisor  of $n$ and
$i$ is $1$. By Proposition~\ref{iso}, we can assume that $i=1$ up to
isomorphism. Now $\pi(a)=\pi(a^{-1})=\pi(a^kb)=\pi(a^{k_1}b)=-2$ and
$\pi(b)=\pi(a^{k_2}b)=1$. By Corollary~\ref{non-bal-ref-auto}, there
exists an automorphism $\psi$ of $D_n$ such that $\psi(a)=a^{-1}$
and $\psi(b)=a^kb$, and this automorphism $\psi$ sends $a^{k_1}b$ to
$a^{k_2}b$. This means that $k_2 = -k_1 +k$. Since $k_1=k_2-k$, we
have $2k_1 \equiv 0 (\mod n)$ and hence $n$ is even and $k_1 =
\frac{n}{2}$. Since
$$\varphi(a^{\frac{n}{2}}bb)= \varphi(a^{\frac{n}{2}}b)\varphi^{-2}(b)=a^{-1}a^{\frac{n}{2}}b \ \ \mbox{and}  \ \   \varphi(ba^{\frac{n}{2}}b)=  \varphi(b)\varphi(a^{\frac{n}{2}}b)=a^kb a^{-1},$$
we have $k=\frac{n}{2}-2$, which means $k_2 = -2$ and hence $p = (b,
a^{\frac{n}{2}-2}b, a, a^{-2}b, a^{\frac{n}{2}}b, a^{-1})$. Since
$\pi(b)=\pi( a^{-2}b ) = 1$, we have $\ker(\pi) =
\langle a^2, b \rangle$. By the fact $a^{\frac{n}{2}}b \notin
\ker(\pi)$, we know that $\frac{n}{2}$ is odd, so $n \equiv 2 (\mod
4)$, and hence ${\cal M}$ is isomorphic to  ${\cal M}_6$.
\end{proof}

In the following lemma, we will show that there is no other reflexible regular Cayley map for $D_n$ except all Cayley maps listed in Theorem~\ref{main-result}.

\begin{lemma} \label{r>3}
There is no reflexible  regular  Cayley map for the
dihedral group $D_n$ whose rotation index $r({\cal M})$ is greater
than $3$.
\end{lemma}

\begin{proof}
Suppose that there exists a reflexible  regular  Cayley
map ${\cal M} = CM(D_n, X, p)$ whose rotation index
$r({\cal M})$ is greater than $3$. For our convenience,
let $r({\cal M})=r$ and let  $x_r = x_0^{-1}=a^i$ and for any
$j=1,\ldots, r-1$, $x_j \in B_n$. For any $j=1, \ldots, r-1$, let
$x_j = a^{k_j}b$.  Now
$\pi(a^{k_1}b)=\pi(a^{k_2}b)=\cdots=\pi(a^{k_{r-2}}b)=1$.   By
Corollary~\ref{non-bal-ref-auto}, there  exists an automorphism
$\psi$ of $D_n$ such that $\psi(a)=a^{-1}$ and
$\psi(a^{k_j}b)=a^{k_{r-j}}b$ for any $j=1, \ldots, r-1$. This
implies that $k_1+k_{r-1} \equiv k_2+k_{r-2} (\mod n)$, and thus
$k_1 -k_2 \equiv k_{r-2}-k_{r-1} (\mod n)$.  Since
\begin{eqnarray*} \varphi(a^{k_1}ba^{k_2}b)&=& \varphi(a^{k_1}b)\varphi(a^{k_2}b)=a^{k_2}ba^{k_3}b \ \ \mbox{and}  \\
\varphi(a^{k_{r-2}}ba^{k_{r-1}}b) &=&
\varphi(a^{k_{r-2}}b)\varphi(a^{k_{r-1}}b)=a^{k_{r-1}}ba^i
\end{eqnarray*}  and  $k_1 -k_2 \equiv k_{r-2}-k_{r-1} (\mod n)$, we
have $a^{k_2}ba^{k_3}b = a^{k_{r-1}}ba^i$, a contradiction.
Therefore there is no reflexible  regular Cayley map
for the dihedral group $D_n$ with $r({\cal M}) \ge 4$.
\end{proof}

\end{document}